\documentclass[11pt,a4paper]{article}
\usepackage{amsmath}
\usepackage{amssymb}
\usepackage{amscd}        
\usepackage{graphics}
\usepackage{color}
\usepackage[all]{xy}
\usepackage{amsthm}   
\usepackage{natbib}
\usepackage{mathrsfs}
\numberwithin{equation}{section}
\theoremstyle{plain}
\newtheorem{theorem}{Theorem}[section]
\newtheorem{lemma}[theorem]{Lemma}

\newtheorem{proposition}[theorem]{Proposition}
\theoremstyle{definition}
\newtheorem{definition}[theorem]{Definition}
\newtheorem{example}[theorem]{Example}
\theoremstyle{remark}
\newtheorem{remark}[theorem]{Remark}
\newcommand{\RQMOD}{\ensuremath{{RQ\text{-}\!\mathcal{M}od}}} 
\newcommand{\RMOD}{\ensuremath{{R\text{-}\!\mathcal{M}od}}} 
\newcommand{\Z}{\ensuremath{\mathbb{Z}}}  
\newcommand{\Q}{\ensuremath{\mathbb{Q}}}

\newcommand{\Hom}{\operatorname{Hom}}
\newcommand{\Image}{\operatorname{Im}} 
\newcommand{\Ker}{\operatorname{Ker}}

\newcommand{\A}{\ensuremath{A}}
\newcommand{\X}{\ensuremath{X}}
\newcommand{\Y}{\ensuremath{Y}}

\newcommand{\coclosed}[2]{\ensuremath{\xymatrix@1{ #1 \, \ar@{^{(}->}[r]^-{cc} &
#2}}}
\newcommand{\cosmall}[3]{\ensuremath{\xymatrix@1{ #1 \,
\ar@{^{(}->}[r]^-{cs}_-{#2} & #3}}}

\newcommand{\ShortExactSequence}[5]{\ensuremath{\xymatrix@1{ 0 \ar[r] &  #1 \ar[r]^-{#2} & #3 \ar[r]^-{#4} &  #5
\ar[r] & 0 }}}

 \newcommand{\Ext}{\operatorname{Ext}} 
 \newcommand{\To}{\longrightarrow}   

\newcommand{\T}{\mathcal{T}}
\newcommand{\F}{\mathcal{F}}
\newcommand{\ZMOD}{\ensuremath{{\Z\text{-}\!\mathcal{M}od}}} 

\begin{document}
\pagestyle{myheadings}
\enlargethispage{\baselineskip}
\title{\vspace{-2cm} Relative homological algebra in categories of representations of infinite quivers}


\author{Sergio Estrada\footnote{Departamento de Matem\'{a}tica Aplicada, Facultad de Inform\'{a}tica, Universidad de Murcia, Murcia, Spain.
e-mail: \texttt{sestrada@um.es}.} \footnote{The author has been
partially supported by DGI MTM2008-03339, by the Fundaci\'on Seneca
07552-GERM and by the Junta de Andaluc\'{\i}a. },
 \hspace{0.3em}  Salahattin \"{O}zdemir\footnote{
            Dokuz Eyl\"{u}l \"{U}niversitesi,
            Fen-Edebiyat Fak\"{u}ltesi,  Matematik B\"{o}l\"{u}m\"{u}, Buca, \.Izmir, Turkey.
            e-mail: \texttt{salahattin.ozdemir@deu.edu.tr}.} \footnote{The author has been supported by The Council of Higher Education (Y\"{O}K) and
 by The Scientific Technological Research Council of Turkey (T\"{U}B\.{I}TAK). }}

 \maketitle
\renewcommand{\theenumi}{\arabic{enumi}}
\renewcommand{\labelenumi}{\emph{(\theenumi)}}

\begin{abstract}
In the first part of this paper, we prove the existence of torsion
free covers in the category of representations of quivers, $(Q,
\RMOD)$, for a wide class of quivers included in the class of the
so-called source injective representation quivers provided that any
direct sum of torsion free and injective $R$-modules is injective.
In the second part, we prove the existence of
$\mathscr{F}_{cw}$-covers and $\mathscr{F}_{cw}^{\perp}$-envelopes
for any quiver $Q$ and any ring $R$ with unity, where
$\mathscr{F}_{cw}$ is the class of all ``componentwise" flat
representations of $Q$. \vspace{0.5cm}\\ \emph{Key Words:} cover;
envelope; torsion free; flat; representations of a quiver.
\vspace{0.5cm}
\\ \emph{2000 Mathematics Subject Classification:} Primary 16G20;
Secondary 16D90

\end{abstract}

\section{Introduction}

The aim of this paper is to continue with a program initiated in
\citet{Enochs-Herzog:HomotopyOfQuiverMorphismsWithApplicationsToRepresentations}
and continued in
\citet{Enochs-et.al:FlatCoversofRepresentationsOfTheQuiverA},
\citet{Enochs-et.al:FlatCoversAndFlatRepresentationsOfOuivers},
\citet{Enochs-Estrada:ProjectiveRepresentationsOfQuivers} and
\citet{Enochs-et.al:InjectiveRepresentationsOfQuivers} to develop
new techniques on the study of representations by modules over
(possibly infinite) quivers. In contrast to the classical
representation theory of quivers motivated by Gabriel's work
(\citet{Gabriel:UnzerlegbareDarstellungen}), we do not assume that
the base ring is an algebraically closed field and that all vector
spaces involved are finite dimensional.

Techniques on representation theory of infinite quivers have
recently proved to be very useful in leading to simplifications of
proofs as well as the descriptions of objects in related categories.
For instance, in
\citet{Enochs-Estrada:RelativeHomologicalAlgebraInTheCategoryOfQuasi-CoherentSheaves}
it is shown that the category of quasi-coherent sheaves on an
arbitrary scheme is equivalent to a category of representations of a
quiver (with certain modifications on the representations). And this
point of view allows to introduce new versions of homological
algebra in such categories (see
\citet[$\S$5]{Enochs-Estrada:RelativeHomologicalAlgebraInTheCategoryOfQuasi-CoherentSheaves}
and \citet{Enochs-et.al:GeneralizedQuasi-CoherentSheaves}). Infinite
quivers also appear when considering the category of $\Z$-graded
modules over the graded ring $R[x]$. This category is equivalent to
the category of representations over $R$ of the quiver
$A_{\infty}^{\infty}\, \equiv \cdots \to \bullet \to \bullet \to
\bullet \to \cdots .$ And, in general, one can find less trivial
example involving group rings $R[G]$ with the obvious grading.

Our goal on this paper is to introduce new classes in the category
of representations of a (possibly infinite) quiver to compute
(unique up to homotopy) resolutions which give rise to new versions
of homological algebra on it.

The first of such versions turns to Enochs' proof on the existence
of torsion free covers of modules over an integral domain (see
\citet{Enochs:Torsion-freeCoveringModules}) and its subsequent
generalization by Teply and Golan in
\citet{Teply:TorsionfreeInjectiveModules} and
\citet{Golan-Teply:Torsion-freeCovers} to more general torsion
theories in $\RMOD$. In the first part of the paper we prove that
torsion free covers exist for a wide class of quivers included in
the class of the so-called source injective representation quivers
as introduced in
\citet{Enochs-et.al:InjectiveRepresentationsOfQuivers}. This
important class of quivers includes all finite quivers with no
oriented cycles, but also includes infinite line quivers:
$$\xymatrix{ A_{\infty} \, \equiv & \cdots \ar[r] & \bullet \ar[r] & \bullet \ar[r] &  \bullet} ,$$
$$\xymatrix{ A^{\infty} \, \equiv & \bullet \ar[r] & \bullet \ar[r] & \bullet \ar[r] &  \cdots} ,$$
$$\xymatrix{ A_{\infty}^{\infty} \, \equiv & \cdots \ar[r] & \bullet \ar[r] & \bullet \ar[r] &  \cdots} $$
and quiver of pure semisimple type as introduced by Drozdowski and
Simson in \citet{Simson-Drozdowski:QuiversOfPureSemisimpleType}.

On the second part, we will focus on the existence of a version of
relative homological algebra by using the class of componentwise
flat representations in $(Q, \RMOD)$. Recently, it has been proved
by Rump that flat covers do exist on each abelian locally finitely
presented category (see
\citet{Rump:FlatCoversInAbelianAndInNon-abelianCategories}). Here by
``flat" the author means Stenstr\"{o}m's concept of flat object
(\citet{Stenstrom:PurityInFunctorCategories}) in terms of the Theory
of Purity that one can always define in a locally finitely presented
category (see
\citet{Crawley-Boevey:LocallyFinitelyPresentedAdditiveCategories}).
We call such flat objects ``categorical flat". For abelian locally
finitely presented categories with enough projectives, this notion
of ``flatness" is equivalent to be direct limit of certain
projective objects. As $(Q, \RMOD)$ is a locally finitely presented
Grothendieck category with enough projectives we infer by using
Rump's result that $(Q, \RMOD)$ admits ``categorical flat" covers
for any quiver $Q$ and any associative ring $R$ with unity. But
there are categories in which there is a classical notion of
flatness having nothing to do with respect to a Theory of Purity.
This is the case of the notion of ``flatness" in categories of
presheaves or quasi-coherent sheaves, where ``flatness" is more
related with a ``componentwise" notion. Those categories may be
viewed as certain categories of representations of quivers, so we
devote the second part to prove the existence of ``componentwise"
flat covers for any quiver and any ring $R$ with unity. In
particular if $X$ is a topological space, an easy modification of
our techniques can prove the existence of a flat cover (in the
algebraic geometrical sense) for any presheaf on $X$ over $\RMOD$.

\section{Preliminaries}

All rings considered in this paper will be associative with identity
and, unless otherwise specified, not necessarily commutative. The
letter $R$ will usually denote a ring. All $R$-modules are left
unitary modules, and all torsion theories considered for $\RMOD$ are
\emph{hereditary}, that is, the torsion class is closed under
submodules, or equivalently, the torsion free class is closed under
injective envelopes, and \emph{faithful} that is, $R$ is torsion
free. We refer to \citet{Enochs-Jenda:RelativeHomologicalAlgebra}
and
\citet{Simson-et.al:ElementsOfTheRepresentationTheoryOfAssociativeAlgebras}
for any undefined notion on covers and envelopes or quivers used in
the text.

A \emph{quiver} is a directed graph whose edges are called arrows.
As usual we denote a quiver by $Q$ understanding that $Q= (V, E)$
where $V$ is the set of vertices and $E$ the set of arrows. An arrow
of a quiver from a vertex $v_1$ to a vertex $v_2$ is denoted by $a:
v_1 \to v_2$. In this case we write $s(a) = v_1$ the initial
(starting) vertex and $t(a) = v_2$ the terminal (ending) vertex. A
\emph{path} $p$ of a quiver $Q$ is a sequence of arrows $a_n \cdots
a_2 a_1$ with $t(a_i) = s(a_{i+1})$. Thus $s(p) = s(a_1)$ and $t(p)
= t(a_n)$. Two paths $p$ and $q$ can be composed, getting another
path $qp$ (or $pq$) whenever $t(p) = s(q)$ ($t(q) = s(p)$).

A quiver $Q$ may be thought as a category in which the objects are
the vertices of $Q$ and the morphisms are the paths of $Q$.

A representation by modules $\X$ of a given quiver $Q$ is a functor
$\X: Q \To \RMOD$. Such a representation is determined by giving a
module $\X(v)$ to each vertex $v$ of $Q$ and a homomorphism $\X(a):
\X(v_1) \to \X(v_2)$ to each arrow $a: v_1 \to v_2$ of $Q$. A
morphism $\eta$ between two  representations $\X$ and $\Y$ is a
natural transformation, so it will be a family  $\eta_v$ such that
$\Y(a)\circ \eta_{v_1} = \eta_{v_2} \circ \X(a)$ for any arrow
$a:v_1 \to v_2$ of $Q$. Thus, the representations of a quiver $Q$ by
modules over a ring $R$ form a category, denoted by $(Q, \RMOD)$.

For a given quiver $Q$ and a ring $R$, the path ring of $Q$ over
$R$, denoted by $RQ$, is defined as the free left $R$-module, whose
base are the paths $p$ of $Q$, and where the multiplication is the
obvious composition between two paths. This is a ring with enough
idempotents, so in fact it is a ring with local units (see
\citet[Ch.10, \S 49]{Wisbauer:FoundationsofModuleandRingTheory}). We
denote by $\RQMOD$ the category of unital $RQ$-modules (i.e. $_{RQ}
M$ such that $RQM = M$). It is known that $RQ$ is a projective
generator of the category and that the categories $\RQMOD$ and $(Q,
\RMOD)$ are equivalent categories, so  $(Q, \RMOD)$ is Grothendieck
category with enough projectives.

For a given quiver $Q$, one can define a family of projective
generators from an adjoint situation as it is shown  in
\citet{Mitchell:RingsWithSeveralObjects}. For every vertex $v\in V$
and the embedding morphism $\{v\} \subseteq Q$ the family $\{S_v(R)
: v\in V\}$ is a family of projective generators of $Q$ where the
functor $S_v : \RMOD \To (Q, \RMOD)$ is defined in \citet[\S
28]{Mitchell:RingsWithSeveralObjects} as $S_v (M)(w) = \oplus_{Q(v,
w)} M$ where $Q(v, w)$ is the set of paths of $Q$ starting at $v$
and ending at $w$. Then $S_v$ is a left adjoint functor of the
evaluation functor $T_v: (Q, \RMOD) \To \RMOD$ given by $T_v (\X) =
\X(v)$ for any representation $\X\in (Q, \RMOD)$.

Let $\mathscr{F}$ be a class of objects in an abelian category
$\mathscr{A}$. Recall from
\citet{Enochs:InjectiveAndFlatCovers-EnvelopesAndResolvents} that,
an $\mathscr{F}$-precover of an object $C$ is a morphism $\psi: F
\to C$ with $F\in \mathscr{F}$ such that $\Hom (F', F) \to \Hom (F',
C) \to 0$ is exact for every $F' \in \mathscr{F}$. If, moreover,
every morphism $f: F \to F$ such that $\psi \circ f = \psi$ is an
automorphism, then $\psi$ is said to be an $\mathscr{F}$-cover.
$\mathscr{F}$-preenvelopes and $\mathscr{F}$-envelopes are defined
dually.

Throughout this paper, by a representation of a quiver we will mean
a representation by modules over a ring $R$.

During this paper we consider the following properties:
\begin{itemize}

\item[(\textbf{A})] Any direct sum of torsion free and injective modules is injective.
 \item[(\textbf{B})] For each vertex $v$ of $Q$, the set $\{t(a) \mid
s(a)= v \}$ is finite.
\end{itemize}

\section{Torsion free covers in the category of representations relative to a torsion theory}

Throughout this section  $Q$ will be a \emph{source injective
representation quiver} (see \citet[Definition
2.2]{Enochs-et.al:InjectiveRepresentationsOfQuivers}), that is, for
any ring $R$ any \emph{injective} representation $\X$ of $(Q,
\RMOD)$ can be characterized in terms of the following conditions:

\begin{itemize}
    \item[(i)]$\X (v)$ is injective $R$-module, for any vertex $v$
    of $Q$.
    \item[(ii)] For any vertex $v$ the morphism $$\X (v) \To \prod_{s(a)= v} \X
    (t(a))$$ induced by $\X (v) \To \X (t(a))$ is a splitting
    epimorphism.
\end{itemize}

\begin{example}
\begin{enumerate}
    \item Each quiver with a finite number of vertices and without
    oriented cycles is source injective.
    \item The infinite line quivers:
$$\xymatrix{ A_{\infty} \, \equiv & \cdots \ar[r] & \bullet \ar[r] & \bullet \ar[r] &  \bullet} ,$$
$$\xymatrix{ A^{\infty} \, \equiv & \bullet \ar[r] & \bullet \ar[r] & \bullet \ar[r] &  \cdots} ,$$
$$\xymatrix{ A_{\infty}^{\infty} \, \equiv & \cdots \ar[r] & \bullet \ar[r] & \bullet \ar[r] &  \cdots}
$$are source injective representation quivers.
    \item Infinite barren trees are source injective representation
    quivers, where a tree $T$ with root $v$ is said to be \emph{barren} if the number
of vertices $n_i$ of the $i'$th state of $T$ is finite for every
$i\in \mathbb{N}$ and the sequence of positive natural numbers $n_1,
n_2, \ldots$ stabilizes (see \citet[Corollaries
5.4-5.5]{Enochs-et.al:InjectiveRepresentationsOfQuivers}). For
example the tree:

$$
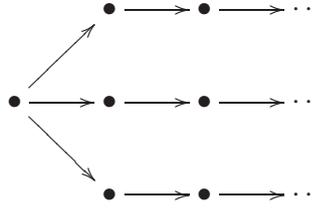
\xymatrix{  &  \bullet \ar[r] &  \bullet \ar[r] & \cdots \\
\bullet \ar[r]  \ar[ur] \ar[dr]  &  \bullet \ar[r] & \bullet \ar[r]
& \cdots
\\ &  \bullet \ar[r] & \bullet \ar[r] & \cdots }$$ is barren.

\item The quiver \, \setlength{\unitlength}{1cm} $\bullet$
\begin{picture}(1,1) \put(0,0){\vector(1,0){1}}
\put(0.4,0.3){$\quad$} \put(0,0.2){\vector(1,0){1}}
\put(0.4,-0.35){$\vdots$} \put(0,0.4){\vector(1,0){1}}
\put(0.6,0.4){$\vdots$}
\end{picture} $\bullet$  $\quad$ is source injective, but does not satisfy
(\textbf{B}).
\end{enumerate}
\end{example}

\begin{example}
The $n$-loop, that is, a loop with $n$ vertices, is not a source
injective representation quiver. To see this, let $v_i$ be a vertex
and $a_i : v_i \to v_{i+1}$ be an arrow of the quiver for all $i= 1,
2, \ldots, n $ where $v_{n+1}= v_1$. Now consider the representation
$\X$ defined as follows: $\X(v_i) = E \times \cdots \times E$ ($n$
times) where $E$ is an injective $R$-module and $\X(a_i)(x_1, \ldots
x_n) = (x_n, x_1, \ldots, x_{n-1})$ where $x_i \in E$ for all $i =
1, \ldots n$. Then it is clear that $\X$ satisfies the conditions
(i) and (ii) of the source injective representations quiver. But
$\X$ is not an injective representation since it is not a divisible
$RQ$-module. This is because, there is a nonzero element $( a_n
a_{n-1}\cdots a_1 + a_1 a_n \cdots a_2 + \cdots + a_{n-1}
a_{n-2}\cdots a_n )- 1$ of $RQ$ such that $$[( a_n a_{n-1}\cdots a_1
+ a_1 a_n \cdots a_2 + \cdots + a_{n-1} a_{n-2}\cdots a_n )- 1]\cdot
m = 0$$ for every element $m = (m_1, \ldots, m_n)$ where $m_i \in  E
\times \cdots \times E$ for all $i=1, 2, \ldots, n$ (notice that if
$x \in \X$ then $x \in E \times \cdots \times E $).
\end{example}

 We recall the definition of a torsion theory.
\begin{definition}\citep{Dickson:TorsionTheoryForAbelianCategories}
A \emph{torsion theory} for an abelian category $\mathcal{C}$ is a
pair $(\T, \F)$ of classes of objects of $\mathcal{C}$ such that
\begin{enumerate}
    \item[(1)] $\Hom (T, F) = 0$ for all $T\in \T$, $F\in \F$.
    \item[(2)] If $\Hom (C, F) = 0$ for all $F\in \F$, then $C\in
    \T$.
    \item[(3)] If $\Hom (T, C) = 0$ for all $T \in \T$, then $C\in \F$.
\end{enumerate}
Here, $\T$ is called \emph{torsion class} and its objects are
\emph{torsion}, while $\F$ is called \emph{torsion free class} and
its objects are \emph{torsion free}. A torsion theory is called
\emph{hereditary} if the torsion class is closed under subobjects.
\end{definition}

Now let $(\T, \F)$ be a torsion theory for $\RMOD$. Then we can
define a torsion theory $(\mathcal{T}_{cw}, \mathcal{F}_{cw})$ for
$(Q, \RMOD)$, by defining the torsion class such that $\X \in
\T_{cw}$  if and only if $\X (v)\in \T$ for all $v\in V$. This is
because $\T_{cw}$ is closed under quotient representations, direct
sums and extensions (as so is $\T$) (see, for example, \citet[VI,
Proposition 2.1]{Stenstrom:RingsOfQuotients}).

\begin{remark}
Since the torsion class $\T_{cw}$ is closed under
subrepresentations, the torsion theory $(\T_{cw}, \F_{cw})$ is
hereditary.
\end{remark}

\begin{proposition}\label{prop:ArepresentationTorsion-freeIFFeachModuleTORSION-free}
Let $\X \in (Q, \RMOD)$. Then  $\X \in \F_{cw}$  if and only if $\X
(v) \in \F$ for all $v\in V$.
\end{proposition}
\begin{proof}
$(\Rightarrow)$ Let $\X \in \F_{cw}$. Then for any $M \in \T$, we
have
$$\Hom_R (M, T_v(\X)) \cong \Hom_Q (S_v (M), \X) = 0$$ since $S_v (M) \in \T_{cw}$ (as $\T$ is closed under direct sums). Thus $\X (v)= T_v(\X) \in \F$ for all $v\in V$.

$(\Leftarrow)$ Suppose that $\X(v) \in \F$ for all $v\in V$. Let $\A
\in \T_{cw}$. If $\gamma : \A \To \X$ is a morphism of
representations, then we have module homomorphisms $\gamma_v : \A
(v) \to \X (v)$ for all $v\in V$. Since $\A(v) \in \T$, then $\Hom_R
(\A (v), \X (v))=0$ and so $\gamma _v = 0$ for all $v\in V$. Thus
$\gamma =0$, that is, $\Hom_Q (\A, \X) = 0$. This means that $\X \in
\F_{cw}$.

\end{proof}

\begin{theorem}\label{thm:thereAreEnoughTorsion-freeInjectives}
Any representation of $\F_{cw}$ can be embedded in a torsion free
and injective representation.
\end{theorem}

\begin{proof}
Let $\X \in \F_{cw}$ be any representation of $Q$. Since $(Q,
\RMOD)$ has enough injectives and $(\T_{cw}, \F_{cw})$ is
hereditary, then $\F_{cw}$ is closed under injective envelopes (see
\citet[Theorem 2.9]{Dickson:TorsionTheoryForAbelianCategories}).
Thus $\X$ can be embedded in its torsion free injective envelope.


\end{proof}


\begin{lemma}\label{lemma:WeCanTakeInjectiveRep.ToBeCovered}
Let $\X, \X', \Y$ and $Z$ be representations of $Q$. Then
\begin{itemize}
\item[(i)] If $\X$ has an $\F_{cw}$-precover and
$Z \subseteq \X$, then $Z$ also has an $\F_{cw}$-precover.
\item[(ii)] If $\X$ is injective, then $\psi: \X' \To \X$ is an $\F_{cw}$-precover of $\X$ if and only if for every morphism
$\phi : \Y \To \X$ with $\Y
\in \F_{cw}$ and $\Y$ injective, there exists $f:\Y \To \X'$ such
that $\psi \circ f = \phi$.
\end{itemize}
\end{lemma}
\begin{proof}
\begin{itemize}
    \item[(i)]Let $\psi: \X' \To \X$ be an $\F_{cw}$-precover.
    Consider the morphism $\psi_1 : \psi^{-1}(Z) \To
    Z$. Then $\psi^{-1}(Z)\in \F_{cw}$ since $\F_{cw}$ is closed under subrepresentations. Now for any morphism $\phi : \Y \To Z$
    with $\Y \in \F_{cw}$, there is a
    morphism $f: \Y \To \X'$ such that $\psi f = \phi$. Therefore, $f(\Y) \subset
    \psi^{-1}(Z)$ and so $\phi$ can be factored through
    $\psi_1$.

    \item[(ii)] The condition is clearly necessary. Let $\phi_1 : \Y_1 \To \X$ be a morphism with $\Y_1 \in \F_{cw}$.
    Then by
    Theorem \ref{thm:thereAreEnoughTorsion-freeInjectives}, $\Y_1$ can be embedded in a representation $\Y \in \F_{cw}$ which is injective. Now since $\X$ is injective, there is a morphism $\phi :\Y \To
    \X$ such that $\phi \mid_{\Y_1} = \phi_1$. So, by hypothesis,
    there exists a morphism $f: \Y \To \X'$ such that $\psi f =
    \phi$. It follows that $(\psi f)\mid_{\Y_1} = \phi \mid_{\Y_1} = \phi_1$.
\end{itemize}
\end{proof}

\begin{lemma}\label{lemma:if-any-direct-sumOf-modulesIs-injectiveThen-so-is-Direct-union}
 Let $E$ be an $R$-module and let $\{E_i\}_{i\in I}$ be a direct family of submodules of $E$. If $\oplus_{i\in I}
 E_i$ is injective, then $\sum_{i\in I} E_i$ is injective.
\end{lemma}
\begin{proof}
Let $\varphi : \oplus E_i \to \sum E_i$ and $\psi : \sum E_i \to
\oplus E_i$ be homomorphisms of $R$-modules such that $\varphi \psi
= id$. Now for any ideal $A$ of $R$, any map $f: A \to \sum E_i$ can
be extended to $\varphi \circ h : R \to \sum E_i$ :
$$\xymatrix{A \ar@{^{(}->}[r] \ar[d]_f & R \ar@{-->}[dl]_{\varphi h} \ar@{-->}[ddl]^h  \\ \sum E_i \ar[d]^{\psi} &   \\ \oplus E_i \ar@/^/@{-->}[u]^{\varphi}  &   }$$
where $(\varphi \circ h )\mid_A = \varphi \circ(h \mid_A) = \varphi
\circ \psi \circ f = f$. Thus $\sum_{i\in I} E_i$ is injective.
\end{proof}

\begin{lemma}\label{lemma:direct-union-of-torsion-freeANDinjective-rep.is-torsion-free-injective}
Let $E$ be in $(Q, \RMOD)$ and let $\{E_i\}_{i\in I}$ be  a direct
family of injective subrepresentations of $E$ such that $E_i
\in\F_{cw}, \forall i \in I$. If $R$ satisfies (\textbf{A}) and if
$Q$ satisfies (\textbf{B}), then $\sum_{i\in I} E_i \in \F_{cw}$ and
it is injective.
\end{lemma}

\begin{proof}
Since each $E_i$ is an injective representation such that $E_i \in
\F_{cw}$, then $E_i (v)$ is an injective module such that $E_i
(v)\in \F$,
 $\forall v\in V$ and  $\forall i \in I$. So $\oplus_{i\in I} E_i(v)$ is also an
injective module by hypothesis. By Lemma
\ref{lemma:if-any-direct-sumOf-modulesIs-injectiveThen-so-is-Direct-union},
 $\sum_{i\in I} E_i(v) $ is also injective. Then the
representation $\sum_{i\in I} E_i$ satisfies (i). Now taking the
union of the splitting epimorphisms $E_i (v) \to \prod_{s(a)= v} E_i
(t(a))$, we obtain the following splitting epimorphism:
$$\bigg(\sum_{i\in I} E_i\bigg)(v) \To \sum_{i\in I}
\prod_{s(a)= v} E_i (t(a))\cong \prod_{s(a)= v} \bigg(\sum_{i\in I}
E_i\bigg) (t(a)) $$ where the isomorphism follows since the product
is finite by hypothesis. This means $\sum_{i\in I} E_i$ is also
satisfies (ii). Thus it is an injective representation since $Q$ is
a source injective representation quiver. Finally, since $E_i(v) \in
\F$ then $\sum_{i\in I}E_i (v) \in \F$ for all $v\in V$, and so
$\sum_{i\in I} E_i \in \F_{cw}$.

\end{proof}

\begin{proposition}\label{prop:Qsatisfy(A)IFF-R-satisfy}
Let $Q$ be any quiver satisfying (\textbf{B}). Then $R$ satisfies
(\textbf{A}) if and only if  any direct sum of injective
representations of $\F_{cw}$ is injective.
\end{proposition}
\begin{proof}
$(\Rightarrow)$ The proof is the same as the proof of Lemma
\ref{lemma:direct-union-of-torsion-freeANDinjective-rep.is-torsion-free-injective}
by taking $\oplus E_i$ instead of $\sum E_i$.

$(\Leftarrow)$ The proof is immediate by considering the quiver $Q
\equiv  \cdot v$ which trivially satisfies (\textbf{B}). This is
because $(Q, \RMOD) \cong \RMOD$ in this case.
\end{proof}

Note that, in the previous proposition, which will be useful in the
proof of the following theorem, we cannot omit the fact that $Q$
satisfy (\textbf{B}) as the following example shows.

\begin{example}\label{exm:QdoesntSatisfy(B)}
Consider the quiver $$
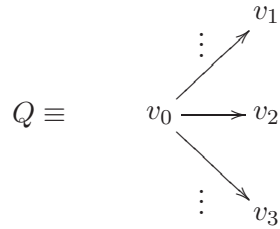
\xymatrix{  &   &  v_1   \\
Q \equiv & v_0 \ar[r] \ar[ur]^{\vdots} \ar[dr]_{\vdots}  &  v_2   \\
 &   & v_3 }$$ which, of course, does not satisfy (\textbf{B})
 for the vertex $v_0$. For the ring of integers, $R = \Z$,
 consider the category $(Q, \ZMOD)$. Then the indecomposable
 injective and
 torsion free representations of $(Q, \ZMOD)$ (w.r.t usual torsion theory) are as follows:

$$
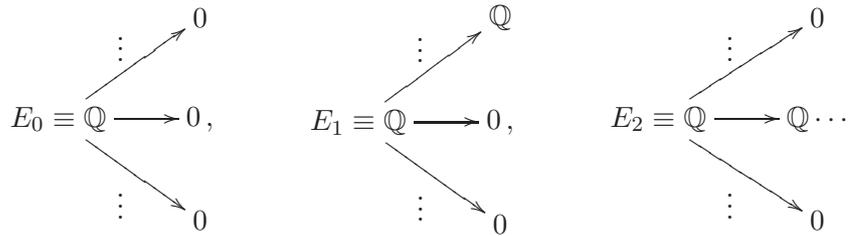
\xymatrix{  &   &  0   \\
 &E_0\equiv  \Q \ar[r] \ar[ur]^{\vdots} \ar[dr]_{\vdots}  &  0 \, , \\
 &   & 0 }   \xymatrix{  &   &  \Q   \\
& E_1 \equiv \Q \ar[r] \ar[ur]^{\vdots} \ar[dr]_{\vdots}  &  0 \, ,  \\
 &   & 0 }  \xymatrix{  &   &  0  \\
&E_2 \equiv  \Q \ar[r] \ar[ur]^{\vdots} \ar[dr]_{\vdots}  &  \Q  \cdots\\
 &   & 0 }$$that is, for each $i\in \mathbb{N}$, the representation $E_i$ has a module $\Q$ at the
 vertices $v_0$ and $v_i$, and $0$ otherwise. Therefore, the direct
 sum of the representations of $E_i$ for $i\geq 1$will be as follows:
 $$
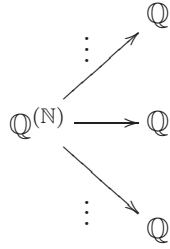
\xymatrix{  &   &  \Q   \\
 &   \Q^{(\mathbb{N})} \ar[r] \ar[ur]^{\vdots} \ar[dr]_{\vdots}  &  \Q \\
 &   & \Q} $$ If we show that $\oplus_{i\geq 1} E_i$ is not an injective representation of $(Q,
 \ZMOD)$, then we will see that the statement of Proposition \ref{prop:Qsatisfy(A)IFF-R-satisfy} does not hold for this $Q$ (since $R=\Z$ satisfies (\textbf{A})). Now suppose on the contrary that $\oplus_{i\geq 1}
 E_i$ is injective. Then since $Q$ is source injective
 representation quiver as it is right rooted (see \citet[Theorem 4.2]{Enochs-et.al:InjectiveRepresentationsOfQuivers}) we have (ii), that is, $$\bigoplus_{i\geq 1}E_i(v_0) \To \prod_{s(a)=v_0} \bigoplus_{i\geq 1} E_i
 (t(a))$$is a splitting epimorphism, or equivalently, $\Q^{(\mathbb{N})} \To
 \Q^{\mathbb{N}}$ is a splitting epimorphism. However, this is
 impossible since $\Q^{(\mathbb{N})}$ has a countable basis but
 $\Q^{\mathbb{N}}$ does not have it.

\end{example}

Now recall that a representation of a quiver $Q$ is said to be
\emph{finitely generated} if it is finitely generated as an object
of the category of representations of $Q$.
\begin{theorem}\label{theorem:if-any-directsum-of-injectives-is-injective-then-any-injective-rep.is-direct-sum-of-indec.injectives}
Let $Q$ be any quiver satisfying (\textbf{B}). If $R$ satisfies
(\textbf{A}), then every injective representation of $\F_{cw}$ is
the direct sum of indecomposable injective representations of
$\F_{cw}$.
\end{theorem}
\begin{proof}
Following the proof of \citet[Proposition
4.5]{Stenstrom:RingsOfQuotients} we argue as follows: let $E \in
\F_{cw}$ be an injective representation of $Q$. Consider all
independent families $(E_i)_{i\in I}$ of indecomposable torsion free
and injective subrepresentations of $E$. Then by Zorn's lemma, there
is a maximal such family $(E_i)_{i\in I}$. Since $\oplus_{i\in I}
E_i\in \F_{cw}$ and it is injective (by Proposition
\ref{prop:Qsatisfy(A)IFF-R-satisfy}), we can write $\displaystyle E
= (\oplus E_i)\oplus E'$. To show that $E' = 0$ it is enough to show
that every injective representation with $0 \neq E' \in \F_{cw}$
contains a non-zero indecomposable direct summand. Consider the set
of all subrepresentations of $E'$ such that:
$$\Sigma = \{E'' \subset E' \mid E''\in \F_{cw}, \text{ injective s.t.
} C \nsubseteq E'' \text{ where } 0 \neq C \subset E' \text{ is
f.g.} \}$$ (In fact, we can take such a non-zero finitely generated
representation $C$, since$(Q, \RMOD)$ is locally finitely
generated). Now take $ \overline{E} = \sum_{E'' \in \Omega} E''$
where $\Omega$ is a chain of $\Sigma$. Then by Lemma
\ref{lemma:direct-union-of-torsion-freeANDinjective-rep.is-torsion-free-injective},
$\overline{E} \in \F_{cw}$ and it is injective. Clearly $C
\nsubseteq \overline{E}$ since $C$ is finitely generated (if $C
\subset \overline{E}$ then $C \subset E''$ for some $E'' \in \Omega$
which is impossible). This shows that $\overline{E} \in \Sigma$ and
in fact it is an upper bound of $\Omega$. Then by Zorn's lemma
$\Sigma$ has a maximal element, say $E''$. Now we have $E' = E''
\oplus D$ where $0\neq D$ is an indecomposable representation. For
if $D = D' \oplus D''$ with $D'\neq 0$ and $D'' \neq 0$, then $(E''
+ D') \cap (E'' + D'') = E''$ and so either $C \nsubseteq E'' + D'$
or $C \nsubseteq E'' + D''$ which contradicts the maximality of
$E''$ in $\Sigma$. Hence, every non-zero $E'$ contains an
indecomposable direct summand, which completes the proof.

\end{proof}

\begin{proposition}
Let $Q$ be any quiver satisfying (\textbf{B}). If $R$ satisfies
(\textbf{A}), then $(Q, \RMOD)$ admits $\F_{cw}$-precovers.
\end{proposition}

\begin{proof}
Since the category $(Q, \RMOD)$ has enough injectives, it suffices
to show that an \emph{injective} representation $\X$ has an
$\F_{cw}$-precover (by Lemma
\ref{lemma:WeCanTakeInjectiveRep.ToBeCovered}-(i)), and so we can
take an injective representation $\Y \in \F_{cw}$ (by Lemma
\ref{lemma:WeCanTakeInjectiveRep.ToBeCovered}-(ii)). Let $\{E_{\mu}
\mid \mu \in \Lambda \}$ denote the set of representatives of
indecomposable injective representations of $\F_{cw}$. Let $H_{\mu}
= \Hom_Q (E_{\mu}, \X)$ and then define $\displaystyle \X' =
\oplus_{\mu \in \Lambda} E_{\mu}^{(H_{\mu})}$. So there is a
morphism $\psi : \X' \To \X$ such that $\psi \mid_{E_{\mu}} \in
H_{\mu}$. Thus every morphism $\phi : \Y \To \X$ with an injective
representation $\Y \in \F_{cw}$ factors through the canonical map
$\psi: \X' \To \X$, since $\Y = \oplus_{\mu \in \Lambda'} E_{\mu}$
by Theorem
\ref{theorem:if-any-directsum-of-injectives-is-injective-then-any-injective-rep.is-direct-sum-of-indec.injectives}
where $\Lambda' \subseteq \Lambda$.
\end{proof}


To prove that $(Q, \RMOD)$ admits $\F_{cw}$-covers, we need the
following lemmas by the same methods of proofs given in, for
example, \citet[Lemmas 1.3.6-1.3.7]{Xu:FlatCoversOfModules} for
usual torsion theories for $\RMOD$.

\begin{lemma}\label{lemma:torsion-freePrecoverIs-torsion-freeCover1}
Let $Q$ be any quiver satisfying (\textbf{B}) and let $R$ satisfy
(\textbf{A}). If $\psi : \X' \To \X$ is an $\F_{cw}$-precover of the
representation $\X$, then we can derive an $\F_{cw}$-precover $\phi
: \Y \To \X$ such that there is no non-trivial subrepresentation $S
\subseteq \ker (\phi)$ with $\Y/S \in \F_{cw}$.
\end{lemma}


\begin{lemma}\label{lemma:torsion-freePrecoverIs-torsion-freeCover2}
Let $Q$ be any quiver satisfying (\textbf{B}) and let $R$ satisfy
(\textbf{A}). If $\phi : \Y \To \X$ is an $\F_{cw}$-precover of $\X$
with no non-trivial subrepresentation $S \subseteq \Y$ such that $S
\subseteq \ker(\phi)$ and $\Y/S \in \F_{cw}$, then this
$\F_{cw}$-precover is actually an $\F_{cw}$-cover of $\X$.
\end{lemma}

\begin{theorem}\label{thm:torsion-free-coversEXISTSfor-quivers}
Let $Q$ be any quiver satisfying (\textbf{B}) and let $R$ satisfy
(\textbf{A}). Then every representation in $(Q, \RMOD)$ has a
unique, up to isomorphism, $\F_{cw}$-cover.
\end{theorem}

\begin{proof}
The existence part of the proof  follows by Lemmas
\ref{lemma:torsion-freePrecoverIs-torsion-freeCover1} and
\ref{lemma:torsion-freePrecoverIs-torsion-freeCover2}, and the
uniqueness part follows by \citet[Theorem
1.2.6]{Xu:FlatCoversOfModules}.
\end{proof}

\begin{example}\label{exm:torsion-freeCover}
Let $R$ satisfy (\textbf{A}). Consider the quiver $Q \equiv \bullet
\to \bullet$. For any module $M$, if we take the torsion free cover
$\psi: T \to M$ of $M$ (this is possible in the category of $\RMOD$,
see \citet{Teply:TorsionfreeInjectiveModules}), then $$\xymatrix{\overline{T} \ar[d]_{\overline{\psi}} \\
\overline{M}}\qquad \xymatrix{ \Ker \psi \ar[r]^{i} \ar[d] & T
\ar[d]^{\psi}
 \\ 0 \ar[r] & M }$$ is an $\F_{cw}$-cover of the representation $0 \to
 M$. In fact, if there is a morphism $$\xymatrix{T_1 \ar[r]^{\alpha}
\ar[d] & T_2 \ar[d]^{\beta}
 \\ 0 \ar[r] & M }$$ where $T_1 \to T_2 \in
 \F_{cw}$, then there exists $f:T_2 \to
 T$ such that $\psi f = \beta$ since $\psi$ is torsion free precover, and so taking $g=f\alpha : T_1 \to \Ker
 \psi$ (it is well-defined since for any $x\in T_1, \psi f \alpha (x) = \beta \alpha (x)=
 0$) we see that it is an $\F_{cw}$-precover. And if there is an
 endomorphism $\overline{f}=(f, g): \overline{T} \to \overline{T}$ such
 that $\overline{\psi}\circ \overline{f} = \overline{\psi}$ then $f$
 is automorphism (since $\psi$ is a torsion free cover), and so $g$
 is a monomorphism. To show that $g$ is epic, take any $y\in \Ker
 \psi$. Then $y= f(x)$ for some $x\in T$ (since $f$ is epic). Since $\psi (x)= \psi f (x) =
 0$, $x \in \Ker \psi$ and thus $y=f(x) = g(x)$ implies that $g$ is
 epic. Hence $\overline{f}$ is an automorphism, that is,
 $\overline{\psi}$ is an $\F_{cw}$-cover.

\end{example}

\begin{remark}
In \citet{Dunkum:AgeneralizationOfBaersLemma}, the question was
raised whether the category ($A^{\infty}, \RMOD$) admits torsion
free covers, where $A^{\infty} \equiv \bullet \to \bullet \to
\cdots$. By Theorem\ref{thm:torsion-free-coversEXISTSfor-quivers},
 if $R$ satisfies (\textbf{A}) then the
 category $(A_{\infty}, \RMOD)$ admits torsion free covers since $A^{\infty}$ satisfies (\textbf{B}).
\end{remark}

\section{Componentwise flat covers in the category of
representations}\label{sec:componentwiseFlatCovers}

Let $\mathscr{A}$ be a Grothendieck category. Recall that an object
$C$ of $\mathscr{A}$ is \emph{finitely presented} if it is finitely
generated and every epimorphism $B\to C$, where $B$ is finitely
generated, has a finitely generated kernel. $\mathscr{A}$ is said to
be a \emph{locally finitely presented} category if it has a family
of finitely presented generators.

In \citet{Rump:FlatCoversInAbelianAndInNon-abelianCategories}, flat
covers are shown to exist in locally finitely presented Grothendieck
categories. Then the category $(Q, \RMOD)$ admits flat covers for
any quiver $Q$ since it is a locally finitely presented Grothendieck
category. This is because $(Q, \RMOD)$ has a family of finitely
generated projective (and so finitely presented) generators. Here by
``flat" we mean \emph{categorical flat} representations of $Q$, that
is, $\displaystyle \lim_{\rightarrow}P_i$ where each $P_i$ is a
projective representation of $Q$.

Now we will define flat representations \emph{componentwise} which
are different from \emph{categorical} flat representations.

\begin{definition} Let $Q$ be any quiver and let $M$ be a representation of $Q$. We call $M$ \emph{componentwise flat} if $M (v)$ is a flat $R$-module for
all $v\in V$.
\end{definition}
This definition is not the categorical definition of flat
representations, but it is the correct one when we consider $(Q,
\RMOD)$ as the category of presheaves over a topological space. By
$\mathscr{F}_{cw}$ we denote the class of all componentwise flat
representations.

Also, let us define \emph{pure subrepresentations} componentwise.
\begin{definition}
An exact sequence of left $R$-modules $$\xymatrix{0 \ar[r] &  A
\ar[r] & B \ar[r]  &  C  \ar[r] &   0}$$ is  \emph{pure} if for
every right $R$-module $L$, the induced sequence $$\xymatrix{0
\ar[r] & L \otimes_{R} A \ar[r] &  L \otimes_{R} B \ar[r]  &  L
\otimes_{R} C \ar[r] &   0 }$$ is exact. A submodule $A \subseteq B$
is \emph{pure} if the induced sequence is pure.
\end{definition}
\begin{definition}
Let $M$ be a representation of $Q$. We call a subrepresentation $P
\subseteq M$  \emph{componentwise pure} if $P(v) \subseteq M(v)$ is
pure submodule for all $v\in V$.
\end{definition}

In the proof of the following lemma, we can consider the
representation generated by  an element ``$x$". Let $M$ be a
representation of $Q$ and let $x\in M$  (so $x\in M(v)$ for some
$v\in V$).  Since $S_v$ is a left adjoint of $T_v$, we have
$$\Hom_{(\{v\}, \RMOD)} (R, M(v)) \cong \Hom_{(Q, \RMOD)} (S_v(R),
M)$$ for all $v\in V$. So we have a unique morphism $\varphi :
S_v(R) \To M$ corresponds to the $R$-homomorphism $\varphi_x:R \to
M(v)$ given by $\varphi_x(1) = x$. Thus $\Image (\varphi)$ is the
subrepresentation of $M$ generated by $x$.

The cardinality of a representation $M$ of a quiver $Q$ is defined
as\\
$\displaystyle |M|= \bigg|\coprod_{v\in V} M(v)\bigg|$.
\begin{lemma}\label{lemma:ExistingPureSubrepresentationForInfiniteCardinal}
Let $\aleph$ be an infinite cardinal such that $\aleph \geq \sup
\{|R|, |V|, |E| \}$. Let $M$ be a representation of $Q$. Then for
each $x\in M$, there exists a componentwise pure subrepresentation
$P$ of $M$ such that $\mid P \mid \leq \aleph$ and $x \in P$.
\end{lemma}
\begin{proof}
Let $x \in M(v)$ with $v\in V$. Then consider the subrepresentation
$M^0 \subseteq M$ generated by $x$. Then $|M^0| \leq \aleph$ since
$$|S_v (R)(w)| = |\oplus_{Q(v, w)}R| \leq |V|\cdot|E|\cdot|\mathbb{N}|\cdot|R|
\leq \aleph \cdot \aleph_0 = \aleph .$$  Since $|M^0 (v)| \leq
\aleph$ for all $v\in V$, we can apply \citet[Lemma
2.5.2]{Xu:FlatCoversOfModules}, so there exist pure submodules $M^1
(v)$ of $M(v)$ such that $|M^1 (v)| \leq \aleph$ and $M^0 (v)
\subseteq M^1 (v)$, $\forall v\in V$. Now consider the
subrepresentation $M^2$ of $M$ generated by $M^1(v)$ such that
$M^1(v) \subseteq M^2 (v)$ for all $v\in V$. Then $ |M^2| =
|\coprod_{v\in V} M^2(v)| = |V|\cdot|M^2(v)|\leq \aleph$ since
$|M^2(v)|\leq \aleph$  as $|M^1(v)|\leq \aleph$ for all $v\in V$. So
applying \citet[Lemma 2.5.2]{Xu:FlatCoversOfModules} again, there
exist pure submodules $M^3(v)$ of $M(v)$ such that $|M^3(v)|\leq
\aleph$ and $M^2(v)\subseteq M^3(v)$ for all $v\in V$. Now consider
the subrepresentation $M^4$ of $M$ generated by $M^3(v)$ such that
$M^3(v)\subseteq M^4(v)$. Then $|M^4| \leq \aleph$. So proceed by
induction to find a chain of subrepresentations of $M$: $M^0
\subseteq M^1 \subseteq M^2 \subseteq \cdots $ such that $|M^n|\leq
\aleph$ for every $n\in \mathbb{N}$ and $v\in V$. Therefore, by
taking $P = \bigcup_{n< \omega} M^n$ we obtain a pure
subrepresentation $P$ of $M$ which satisfies the hypothesis of the
lemma. Indeed, $P$ is componentwise pure subrepresentation of $M$,
because for each $v\in V$, the set $\{n\in \mathbb{N} : M^n (v)
\text{ is pure in } M(v)\}$ is cofinal, and the set $\{n\in
\mathbb{N} : M^n \text{ is a subrepresentation of } M\}$ is also
cofinal. Finally, it is clear that $|P| \leq \aleph_0 \cdot \aleph =
\aleph$, and $x\in P$ since $x\in M^0 (v)$.
\end{proof}

We recall that a chain of subobjects of a given object $C$ of an
abelian category, $\{C_{\alpha} : \alpha < \lambda\}$ where
$\lambda$ is an ordinal number, is said to be \emph{continuous}
provided that $C_{\omega} = \cup_{\alpha < \omega} C_{\alpha}$ for
any limit ordinal $\omega< \lambda$.

Given a class $\mathscr{F}$ of objects in an abelian category
$\mathscr{A}$, by $\mathscr{F}^{\perp}$ we denote the class of
objects $C$ of $\mathscr{A}$ such that $\Ext^1(F, C) = 0 $ for all
$F \in \mathscr{F}$, and  we say that $(\mathscr{F},
\mathscr{F}^{\perp})$ is cogenerated by a set if there exists a set
$T \subseteq \mathscr{F}$ such that $T^{\perp}= \mathscr{F}^{\perp}$
(see, for example,
\citet[Chp.7]{Enochs-Jenda:RelativeHomologicalAlgebra}).

\begin{theorem}\label{thm:(F,F^p)cogenerated-by-set}
The pair $(\mathscr{F}_{cw}, \mathscr{F}_{cw}^{\perp})$ is
cogenerated by a set.
\end{theorem}
\begin{proof}
Let $F \in \mathscr{F}_{cw}$ and take any element $x_0 \in F$. Then
by Lemma
\ref{lemma:ExistingPureSubrepresentationForInfiniteCardinal}, there
exists a componentwise pure subrepresentation $F_0 \subseteq F$ such
that $x_0\in F_0$ and $|F_0|\leq \aleph$ for a suitable cardinal
number. Since a pure submodule of a flat module is flat, then $F_0
\in \mathscr{F}_{cw}$, and so $F/F_0 \in \mathscr{F}_{cw}$. Then
take any element $x_1 \in F/F_0$ and find a componentwise pure (and
so componentwise flat) subrepresentation $F_1/F_0 \subseteq F/F_0$
such that
 $x_1 \in F_1/F_0$ and $|F_1/F_0| \leq \aleph$. Since $F_0, F_1/F_0 \in
 \mathscr{F}_{cw}$, we have $F_1\in \mathscr{F}_{cw}$ and so $F/F_1 \in \mathscr{F}_{cw}$. Now take $x_2 \in F/F_1$ and, since $\mathscr{F}_{cw}$ is closed under direct limits, proceed by
 transfinite induction to find, when $\alpha$ is a successor
 ordinal, subrepresentations $F_{\alpha} \subseteq F$ such that $F_{\alpha}/F_{\alpha-1}\in
 \mathscr{F}_{cw}$ (and so $F_{\alpha} \in \mathscr{F}_{cw}$) and that $|F_{\alpha}/F_{\alpha-1}| \leq
 \aleph$. When $\omega$ is a limit ordinal, define $F_{\omega} = \bigcup_{\alpha < \omega}
 F_{\alpha}$. So  $F_{\omega}\in \mathscr{F}_{cw}$  and $|F_{\omega}| \leq \aleph$ for every
 $\omega$. Now there exists an ordinal $\lambda$ such that $F$ is a
 direct union of the continuous chain $\{F_{\alpha} \mid \alpha <
 \lambda\}$ where by construction $F_0, F_{\alpha+1}/F_{\alpha} \in
 \mathscr{F}_{cw}$ and $|F_0|\leq \aleph, |F_{\alpha+1}/F_{\alpha}| \leq \aleph$.
 Thus if we choose a set $T$ of representatives of all componentwise flat
 representations with cardinality less than or equal to $\aleph$,
 then  by \citet[Lemma 1]{Eklof-Trlifaj:HowToMakeExtVanish}, we see that the pair $(\mathscr{F}_{cw}, \mathscr{F}_{cw}^{\perp})$ is cogenerated by
 $T$ (note that \citet[Lemma 1]{Eklof-Trlifaj:HowToMakeExtVanish} is for
module categories, but the same arguments of the proof carry over
general Grothendieck categories).

\end{proof}

\begin{theorem}\label{thm:AnyRep.hasFC-cover}
For any quiver $Q$, any representation of $Q$ has an
$\mathscr{F}_{cw}$-cover and an $\mathscr{F}_{cw}^{\perp}$-envelope.
\end{theorem}
\begin{proof}
It is clear that $\mathscr{F}_{cw}$ is closed under direct sums,
extensions and well ordered direct limits (as so is the class of all
flat modules). Moreover,  $S_v (R)(w) = \oplus_{Q(v, w)}R$ is a
projective (and so flat) module for all $w \in V$. Therefore $S_v(R)
\in \mathscr{F}_{cw}$. Now, apply Theorem 2.6 in
\citet{Enochs-et.al:FlatCoversOfQuasiCoherentSheavesOverProjectiveLine}
with Theorem \ref{thm:(F,F^p)cogenerated-by-set} to obtain the
result.
\end{proof}

Recall that a commutative integral domain is called \emph{Pr\"{u}fer
domain} if every finitely generated ideal is projective. Over such a
domain a module is flat if and only if it is torsion free (see
\citet{Rotman:HomologicalAlgebra} for the details). Combining this
fact with the previous result, we have that

\begin{theorem}\label{thm:OverPruferDomain-FC-cover=F_cw-cover}
Let $R$ be a Pr\"{u}fer domain. Then every representation in $(Q,
\RMOD)$ has an $\mathscr{F}_{cw}$-cover agreeing with its
$\F_{cw}$-cover.
\end{theorem}
\begin{remark}
In Example \ref{exm:QdoesntSatisfy(B)}, since $Q$ does not satisfy
(\textbf{B}) we cannot use Theorem
\ref{thm:torsion-free-coversEXISTSfor-quivers} to determine whether
$(Q, \ZMOD)$ admits $\F_{cw}$-covers. However, since $R= \Z$ is a
Pr\"{u}fer domain, $(Q, \ZMOD)$ admits $\F_{cw}$-covers by Theorem
\ref{thm:OverPruferDomain-FC-cover=F_cw-cover}.
\end{remark}

\section{Examples of comparing categorical flat covers with $\mathscr{F}_{cw}$-covers}
In this section, we will provide some examples on the different
kinds of covers studied throughout the paper.

The \emph{categorical} flat representations are characterized (for
rooted quivers) in \citet[Theorem
3.7]{Enochs-et.al:FlatCoversAndFlatRepresentationsOfOuivers} as
follows: a representation $F$ of a quiver $Q$ is \emph{flat} if and
only if $F(v)$ is a flat module and the morphism
$\oplus_{t(a)=v}F(s(a)) \to F(v)$ is a pure monomorphism for every
vertex $v\in V$. In this case, as we pointed out at the beginning of
Section \ref{sec:componentwiseFlatCovers}, it is known that $(Q,
\RMOD)$ admits categorical flat covers for any quiver $Q$. Moreover,
we have proved in Theorem \ref{thm:AnyRep.hasFC-cover} that $(Q,
\RMOD)$ also admits $\mathscr{F}_{cw}$-covers (i.e. componentwise
flat covers). In this section, we will give some examples of
categorical flat covers and of $\mathscr{F}_{cw}$-covers showing
that these two kind of covers do not coincide in general.

Recall that a  module $C$ is called \emph{cotorsion} if $\Ext^1 (F,
C) = 0$ for any flat module $F$. Since every module has a flat cover
(\citet{Bican-et.al:AllModulesHaveFlatCovers}), every module has a
cotorsion envelope by \citet[Theorem 3.4.6]{Xu:FlatCoversOfModules}.
\begin{example}\label{exm:firstExmForCat.FlatCovers}

Let $Q$ be the quiver $\bullet \to \bullet$. Let us take any module
$M$ and the flat cover $\varphi : F \to M$ of it. Then:
\begin{enumerate}
    \item
    $$\xymatrix{0 \ar[r]
\ar[d] & F \ar[d]^{\varphi}
 \\ 0 \ar[r] & M} $$ is a flat cover of the representation $0 \to
 M$. To see this, let $$\xymatrix{F_1 \ar[r]^{\alpha}
\ar[d] & F_2 \ar[d]^{\beta}
 \\ 0 \ar[r] & M }$$ be a morphism, where $\alpha$ is a pure
 monomorphism and $F_1, F_2$ are flat modules. Then $F_2 / F_1$ is
 also a flat module. Since $\varphi$ is a flat cover, there exists $\delta : F_2 \to
 F$ such that $\varphi \delta = \beta$. It is clear that $\varphi \delta \alpha = \beta \alpha =
 0$, and so there exists a unique $h: F_1 \to \Ker \varphi$ such
 that $\delta \alpha = i h$, where $i : \Ker \varphi \to F$. From the short exact sequence $0 \to F_1 \to F_2 \to F_2/F_1 \to 0$, we obtain $$\Hom(F_2, \Ker \varphi) \to \Hom(F_1, \Ker \varphi) \to 0 ,$$
 since $\Ext^1(F_2/F_1, \Ker \varphi ) = 0$ by Wakamutsu's lemma (see \citet[Lemma 2.1.1]{Xu:FlatCoversOfModules}). So there is $z: F_2 \to \Ker \varphi$
 such that $z \alpha = h$. Now if we consider $\delta - z : F_2 \to
 F$, then clearly $\varphi(\delta - z) = \beta$ and $(\delta - z)\alpha =
 0$.

    \item If we take the flat cover $f: G\to \Ker \varphi$ of $\Ker \varphi$ then
 $$\xymatrix{G \ar[r]^t
\ar[d] & F \ar[d]^{\varphi}
 \\ 0 \ar[r] & M} $$ is an $\mathscr{F}_{cw}$-cover of the representation $0 \to
 M$. In fact, if $$\xymatrix{F_1 \ar[r]^{\alpha}
\ar[d] & F_2 \ar[d]^{\beta}
 \\ 0 \ar[r] & M }$$ is a morphism where $F_1 \to F_2 \in \mathscr{F}_{cw}$,
 then clearly there exists $h: F_2 \to F$ such that $\varphi h = \beta$, since $\varphi$ is a flat
 cover. Since $\varphi h \alpha = \beta \alpha = 0$, the map $h \alpha : F_1 \to \Ker
 \varphi$ is defined. Then there exists $h' : F_1 \to G$ such that $f h' = h
 \alpha$ since $f$ is a flat cover, and so $h \alpha = t h'$. This
 shows that $\{0, \varphi\}$ is an $\mathscr{F}_{cw}$-precover. To see that
 it is a cover, suppose there is an endomorphism $$\xymatrix{G \ar[r]^t
\ar[d]_g & F \ar[d]^{g'}
 \\ G \ar[r]^t & F }$$ such that $0 g = 0$ and $\varphi g' =
 \varphi$. Then clearly $g'$ is an automorphism since $\varphi$ is a
 flat cover. Now we show that $g$ is also an automorphism. Since $\varphi g' i =
 0$, there exists $\psi : \Ker \varphi \to \Ker \varphi$ where $i:\Ker \varphi \to
 F$. Actually, $\psi$ is an automorphism (see the comment of $g$ being an automorphism in Example
 \ref{exm:torsion-freeCover}), and so from the commutative diagram $$\xymatrix{G \ar[r]^f \ar[d]_g & \Ker \varphi \ar[d]^{\psi}
 \\ G \ar[r]^f & \Ker \varphi }$$ we obtain that $g$ is also an
 automorphism (by using the fact that $f$ is a cover).
\end{enumerate}
\end{example}
\begin{remark}
Note that in the previous example $0 \to F$ cannot be an
$\mathscr{F}_{cw}$-precover of $ 0 \to M$. Because by (2), $G\to F$
is an $\mathscr{F}_{cw}$-cover of $ 0 \to M$ with $G\neq 0$ and it
is known that covers are direct summand of precovers (see
\citet[Theorem 1.2.7]{Xu:FlatCoversOfModules}). So if $0 \to F$ were
an $\mathscr{F}_{cw}$-precover of $0\to M$, then we would have
$$(0 \to F) = (G\to F)\oplus (H_1 \to H_2) = (G\oplus H_1 \to F
\oplus H_2)$$ for some representation $H_1 \to H_2$ of $Q$. This
implies $0 = G \oplus H_1$ which contradicts the fact that $G\neq
0$.
\end{remark}
\begin{remark}
Comparing with Example \ref{exm:torsion-freeCover}; $\Ker \varphi
\to F$ is a torsion free cover but not an $\mathscr{F}_{cw}$-cover
of $0 \to M$ (unless $\Ker \varphi$ is a flat module). Since the
class of torsion free modules is closed under submodules, but the
class of flat modules is not.
\end{remark}

\begin{example}
Let $Q$ be the quiver $\bullet \to \bullet$. Let us take any module
$M$ and the flat cover $\varphi : F \to M$ of it. Then,
$$\xymatrix{F \ar[r]^{id} \ar[d]_{\varphi} & F \ar[d]^{\varphi}
 \\ M \ar[r]^{id} & M} $$ is both a (categorical) flat cover and an $\mathscr{F}_{cw}$-cover of the
representation $\xymatrix{ M \ar[r]^{id} & M}$. In fact, if there is
a morphism $$\xymatrix{F_1 \ar[r]^h \ar[d]_{\psi_1} & F_2
\ar[d]^{\psi_2}
 \\ M \ar[r]^{id} & M} $$ where $F_1, F_2$ are flat modules and $h$ is a pure monomorphism, then clearly
there is $f: F_2 \to F$ such that $\varphi f = \psi_2$ (since
$\varphi$ is a flat cover). Taking $fh : F_1 \to F$, we see that
$\varphi f h = \psi_2 h = \psi_1$. This means that $\xymatrix{ F
\ar[r]^{id} & F}$ is a flat precover, and clearly it is a flat cover
(since $id_F$ is a pure monomorphism). Since we have not used the
fact that $h$ is pure, then $\xymatrix{ F \ar[r]^{id} & F}$ is also
a $\mathscr{F}_{cw}$-cover of $\xymatrix{ M \ar[r]^{id} & M}$.

\end{example}

\begin{example}
Let $Q$ be the quiver $\bullet \to \bullet \to \bullet$ and let $M$
be a module. Let us take the flat cover $\varphi: F \to M$ of $M$.
Then:

\begin{enumerate}
    \item If we take the cotorsion envelope $i: F \to C$ of $F$,
    then $C$ will be a flat module by \citet[Theorem 3.4.2]{Xu:FlatCoversOfModules}). Therefore, we have a (categorical) flat representation $\overline{F} \equiv \xymatrix{F \ar[r]^{i} & C \ar[r]^{k_1} & C \times
    F}$ where $k_1$ is canonical inclusion (since $k_1$ and $i$ are pure monomorphisms). We show that $$\xymatrix{\overline{F} \ar[d]_{\overline{\varphi}} \\ \overline{M}}
    \qquad  \xymatrix{F \ar[r]^{i} \ar[d]^{\varphi} & C \ar[r]^{k_1} \ar[d]^0 & C \times
    F \ar[d]^{\varphi p_2}\\ M \ar[r]& 0 \ar[r] & M}$$is a flat
    cover of the representation $\overline{M}$
    of $Q$, where $p_2: C\times F \to F$ is a projection. In fact,
    if there is a morphism $$\xymatrix{F_1 \ar[r]^{\alpha} \ar[d]^{t_1} & F_2 \ar[r]^{\beta} \ar[d]^0 & F_3 \ar[d]^{t_3}\\ M \ar[r]& 0 \ar[r] & M}$$
with $F_1, F_2, F_3$ are flat modules and $\alpha, \beta$ are pure
monomorphisms, then clearly there exists $f: F_1 \to F$ such that
$\varphi f = t_1$ since $F$ is a flat cover of $M$. From the short
exact sequence $0 \to F_1 \to F_2 \to F_2/F_1 \to 0$, we obtain that
$$\Hom(F_2, C) \to \Hom(F_1, C) \to 0$$ is exact, since $\Ext^1(F_2/F_1, C)=
0$ (as $\alpha$ is a pure monomorphism and $C$ is cotorsion). So
there exists $g:F_2 \to C$ such that $g \alpha = i f$. Now, since
$F$ is a flat cover of $M$, there exists $\tau_2 : F_3 \to F$ such
that $\varphi \tau_2 = t_3$. Moreover, from the short exact sequence
\begin{equation}\label{eqn:shortExactSeq}
0 \to F_2 \to F_3 \to F_3/F_2 \to 0
\end{equation}
we obtain that
$$\Hom(F_3, C) \to \Hom(F_2, C) \to 0$$ is exact. Then there exists $\tau_1: F_3 \to C$ such that
$\tau_1 \beta = g$. Since $\varphi \tau_2 \beta = t_3 \beta = 0$,
there exists a unique $\gamma: F_2 \to \Ker \varphi$ such that
$\gamma = \tau_2 \beta$. Similarly, if we take $\Ker \varphi$
instead of $C$,  by (\ref{eqn:shortExactSeq}), there exists $z : F_3
\to \Ker \varphi$ such that $z \beta = \gamma$. Therefore, by
defining $h:F_3 \to C \times F$ such that $h(x) = \big(\tau_1(x),
(\tau_2-z)(x)\big)$ for all $x\in F_3$, we see that $\varphi p_2 h =
\varphi (\tau_2-z) = t_3$, and moreover $h \beta = \big(\tau_1
\beta, (\tau_2-z) \beta \big) = (g, 0)= k_1 g$ . Thus
$\overline{\varphi}:\overline{F} \to \overline{M}$ is a flat
precover. To see that it is a cover, suppose $s = \{f, g, h\}:
\overline{F} \to \overline{F}$ is an endomorphism such that
$\overline{\varphi} \circ s = \overline{\varphi}$. It is clear that
$f$ and $g$ are automorphisms. For $h: C\times F \to C\times F$, we
set $h_{ij} = \pi_i h e_j$ for $i, j = 1, 2$ where $\pi_k$ is
projection and $e_k$ is injection. We can write $h$ in a matrix form
as
$$\left( \begin{array}{cc}
h_{11} & h_{12} \\
h_{21} & h_{22}
\end{array} \right) .$$ Since $h k_1 = k_1 g$ then $h_{11}=g$ and $h_{21}=
0$, and since $\varphi$ is a cover, $\varphi = \varphi h_{22}$
implies that $h_{22}$ is an automorphism. Hence $h$ is an
automorphism and so is $s$, that is, $\overline{\varphi}$ is a
cover.

    \item If we take the flat cover $f: G \to \Ker \varphi$ of $\Ker \varphi$, then it is
    immediate that $$\xymatrix{F \ar[r]^0 \ar[d]_{\varphi} &  G \ar[r]^t \ar[d]_0 & F \ar[d]_{\varphi}  \\
    M \ar[r]^0 & 0 \ar[r]^0  &  M}$$ is an $\mathscr{F}_{cw}$-cover of the
    representation $M \to 0 \to M$.
\end{enumerate}
\end{example}

\bibliographystyle{deufen}
\bibliography{kaynaklar-Selahattin}
\end{document}